\documentclass[11pt]{article}
\usepackage{amsmath,amsthm,amssymb}
\theoremstyle{plain}
\newtheorem{theorem}{Theorem}[section]
\newtheorem{proposition}[theorem]{Proposition}
\newtheorem{lemma}[theorem]{Lemma}
\newtheorem{corollary}[theorem]{Corollary}
\newtheorem{claim}[theorem]{Claim}

\theoremstyle{definition}
\newtheorem{definition}[theorem]{Definition}
\newtheorem{example}[theorem]{Example}
\theoremstyle{remark}
\newtheorem{remark}[theorem]{Remark}
\numberwithin{equation}{section}

\renewcommand{\eqref}[1]{(\ref{#1})}
\newcommand{\field}[1]{\mathbb{#1}}
\newcommand{\C}{\field{C}}
\newcommand{\R}{\field{R}}
\newcommand{\Z}{\field{Z}}
\newcommand{\K}{\field{K}}
\newcommand{\id}{{\mathrm{id}}}

\begin{document}
\title{Dynamical Yang-Baxter maps
and Hopf algebroids
associated with s-sets}
\author{Noriaki Kamiya\thanks{Center for Mathematical Sciences,
University of Aizu, Aizuwakamatsu, Fukushima 9658580, Japan} 
\and Youichi Shibukawa\thanks{Department of Mathematics,
Faculty of Science, Hokkaido University, Sapporo 0600810, Japan
}}
\date{June 8, 2016}
\maketitle
\begin{abstract}
An s-set is an algebraic generalization of the regular s-manifold
introduced by Kowalski, one of the
generalized symmetric spaces in differential geometry.
We prove that suitable s-sets
give birth to dynamical Yang-Baxter maps,
set-theoretic 
solutions to a version of the quantum dynamical Yang-Baxter equation.
As an application, 
Hopf algebroids and rigid tensor categories are constructed
by means of these dynamical Yang-Baxter maps.
\end{abstract}
\section{Introduction}
The quantum 
Yang-Baxter equation \cite{baxter72,baxter78,yang67,yang68}
is closely related to algebraic structures,
for example, the quantum group \cite{drinfeld,jimbo}, 
the Hopf algebra \cite{abe,montgomery,sweedler}, 
and the triple system \cite{okubokamiya}. 
Analogously, the quantum dynamical Yang-Baxter equation 
\cite{felder,gervais},
a generalization of this equation,
produces Hopf algebroids 
\cite{bohm2005,bohm2004,day,lu,maltsiniotis,ravenel,schauenburg,xu}.
In fact,
Felder's dynamical R-matrix \cite{felder}, a solution to
the quantum dynamical Yang-Baxter equation, yields a Hopf algebroid
called the elliptic quantum group \cite{etingof}
through the Faddeev-Reshetikhin-Takhtajan
construction \cite{faddeev}.

In a similar way,
suitable dynamical Yang-Baxter maps
\cite{s2005,s20101}, 
set-theoretic solutions to
a version of the quantum dynamical Yang-Baxter equation
(see Remark \ref{pre:rem:qdybe}),
produce Hopf algebroids \cite{s2016,stakeuchi}.
This Hopf algebroid
implies two rigid tensor categories,
one of which is a category
consisting of finite-dimensional L-operators.

Our purpose is to find dynamical Yang-Baxter maps
giving birth to Hopf algebroids and rigid tensor categories.

Several studies clarified how to construct
the dynamical Yang-Baxter maps
\cite{matsumoto,matsumotos,s2005,s2007,s20102}. 
In \cite{kamiyas},
suitable homogeneous pre-systems 
produce 
the dynamical Yang-Baxter maps.
The homogeneous pre-system is a generalization of the homogeneous system
\cite{kikkawa} in differential geometry,
an algebraic feature of
the reductive homogeneous
space with suitable conditions.

It is natural to try to relate
geometric structures
to the construction of the dynamical Yang-Baxter map
that can provide with
Hopf algebroids and rigid tensor categories.

Kowalski 
\cite[Definition II.2]{kowalski}
presented a notion of the regular s-manifold,
a generalization of the symmetric space in the sense of \cite{loos}.
The regular s-manifold
is a $C^\infty$-manifold $M$ with a differentiable 
multiplication $M\times M\ni(x, y)\mapsto x\cdot y\in M$
such that
the maps $s_x: M\ni y\mapsto x\cdot y\in M$
satisfy the following:
\begin{enumerate}
\item
$s_x(x)=x$ for any $x\in M$;
\item
every map $s_x$ is a diffeomorphism;
\item
$s_x\circ s_y=s_{s_x(y)}\circ s_x$
for any $x, y\in M$;
\item
for each $x\in M$, the tangent map
${(s_x)_*}{}_x: T_x(M)\to T_x(M)$
has no fixed vector except the null vector.
\end{enumerate}

The aim of this paper is to construct Hopf algebroids
and rigid tensor categories
by introducing a notion of the s-set
(Definition \ref{tos:def:sset}),
a generalization of the regular s-manifold 
from the algebraic point of view.
In this construction, suitable s-sets give birth to 
the dynamical Yang-Baxter maps through 
the homogeneous pre-systems.
This paper gives another way to relate
the dynamical Yang-Baxter map
to differential geometry.

The organization of this paper is as follows.
In Section 2, we give a brief exposition of
the homogeneous pre-system, the dynamical Yang-Baxter map,
the Hopf algebroid,
and the rigid tensor category
\cite{kamiyas,s2016}.
Section 3 discusses the construction of ternary operations
by means of the s-sets.
In Section 4, we apply the results of Section 3 to
get the dynamical Yang-Baxter maps via the homogeneous pre-systems
from suitable s-sets. 
The last section, Section 5, is devoted 
to the study of the Hopf algebroids
associated with the above dynamical Yang-Baxter maps.
Each Hopf algebroid can produce
the rigid tensor category
consisting of finite-dimensional L-operators.
\section{Summary of homogeneous pre-systems, 
dynamical Yang-Baxter maps, Hopf algebroids,
and rigid tensor categories}
In this section, we summarize without proofs the relevant material on
homogeneous pre-systems, dynamical Yang-Baxter maps, Hopf algebroids,
and rigid tensor categories
\cite{kamiyas,s2016},
to render this paper as self-contained as possible.

We first introduce quasigroups
\cite{pflugfelder,smith}.
\begin{definition}\label{pre:def:quasigroup}
A nonempty set $Q$ with a binary operation 
$Q\times Q\ni(u, v)\mapsto uv\in Q$
is called a quasigroup,
iff:
\begin{enumerate}
\item
for any $u, w\in Q$, there uniquely exists $v\in Q$
such that $uv=w$;
\item
for any $v, w\in Q$, there uniquely exists $u\in Q$
such that $uv=w$.
\end{enumerate}
\end{definition}
That is to say, the left and the right translations on the quasigroup
are both bijective.
On account of this fact,
we define the map $\backslash: Q\times Q\to Q$ by
\begin{equation}\label{pre:eq:backslash}
v=u\backslash w
\Leftrightarrow uv=w
\quad(u, v, w\in Q).
\end{equation}

Any group is a quasigroup; on the other hand, the quasigroup 
is not always associative.
\begin{example}[\cite{s2016}]
Let $Q_5:=\{ 0, 1, 2, 3, 4\}$.
We define a binary operation on this set $Q_5$
by Table \ref{pre:table:q5}.
Here $0\,4=0$.
Each element of $Q_5$ appears once and only once in each row and 
in each column of Table \ref{pre:table:q5},
and this set $Q_5$ is hence a quasigroup
\cite[Theorem I.1.3]{pflugfelder}.
The binary operation on $Q_5$ is not associative, because
$(12)3=1\ne4=1(23)$.
\begin{table}[h]
\caption{The binary operation on $Q_5$ \cite{s2016}.}
\label{pre:table:q5}
\begin{center}
\begin{tabular}{r|rrrrr}
\hline
\ &0&1&2&3&4
\\\hline
0&4&3&2&1&0
\\
1&3&1&0&2&4
\\
2&0&2&3&4&1
\\
3&1&0&4&3&2
\\
4&2&4&1&0&3
\\\hline
\end{tabular}
\end{center}
\end{table}
\end{example}
\begin{definition}\label{pre:def:hps}
A pair $(S, \eta)$ of 
a nonempty set $S$ and a ternary operation 
$\eta: S\times S\times S\to S$
is called a 
homogeneous pre-system
\cite{kamiyas},
iff
the ternary operation $\eta$ satisfies:
for any $x, y, u, v, w\in S$,
\begin{align}
\nonumber
&\eta(x, y, x)=y;
\\\label{hps:eq:eta1}
&\eta(x, y, \eta(u, v, w))
=\eta(\eta(x, y, u), \eta(x, y, v), \eta(x, y, w)).
\end{align}
\end{definition}
This homogeneous pre-system
$(S, \eta)$
satisfying 
\begin{equation}\label{pre:eq:hps}
\eta(x, y, z)=\eta(w, \eta(x, y, w), z)
\quad(\forall x, y, z, w\in S),
\end{equation}
together with a suitable quasigroup,
can produce a dynamical Yang-Baxter map
\cite{kamiyas}.
Let $H$ and $X$ be nonempty sets
with a map $H\times X\ni(\lambda, x)\mapsto \lambda x\in H$.
\begin{definition}\label{pre:def:dybm}
A map $\sigma(\lambda): X\times X\to X\times X$
$(\lambda\in H)$ is a dynamical Yang-Baxter map,
iff
$\sigma(\lambda)$ satisfies
a version of the quantum dynamical Yang-Baxter equation
\begin{equation}\label{pre:eq:avbr}
\sigma_{12}(\lambda X^{(3)})\sigma_{23}(\lambda)
\sigma_{12}(\lambda X^{(3)})
=
\sigma_{23}(\lambda)
\sigma_{12}(\lambda X^{(3)})
\sigma_{23}(\lambda)
\quad(\forall\lambda\in H).
\end{equation}
Here, the maps
$\sigma_{12}(\lambda X^{(3)}), 
\sigma_{23}(\lambda): X\times X\times X\to
X\times X\times X$ are
defined by
\[
\sigma_{12}(\lambda X^{(3)})(x, y, z)
=
(\sigma(\lambda z)(x, y), z);
\sigma_{23}(\lambda)(x, y, z)=(x, \sigma(\lambda)(y, z)).
\]
\end{definition}
\begin{remark}\label{pre:rem:qdybe}
For a map
$\sigma(\lambda): X\times X\to X\times X$
$(\lambda\in H)$,
we set
$R(\lambda)(x, y):=\sigma(\lambda)(y, x)$
$(\lambda\in H, x, y\in X)$.
Then the following conditions are equivalent:
\begin{enumerate}
\item
the map $\sigma(\lambda)$ satisfies \eqref{pre:eq:avbr};
\item
the map $R(\lambda)$ satisfies
\begin{equation}\label{pre:eq:avqdybe}
R_{23}(\lambda)R_{13}(\lambda X^{(2)})R_{12}(\lambda)
=
R_{12}(\lambda X^{(3)})R_{13}(\lambda)R_{23}(\lambda X^{(1)})
\end{equation}
for any $\lambda\in H$
\cite[(2.1)]{s2005}.
\end{enumerate}
Throughout this paper,
both \eqref{pre:eq:avbr}
and \eqref{pre:eq:avqdybe} are called
versions of the
quantum dynamical Yang-Baxter equation
(see also Remark \ref{pre:rem:brvg}).
\end{remark}
Let $(S, \eta)$ be a homogeneous pre-system satisfying
\eqref{pre:eq:hps},
and let $Q$ be a quasigroup, isomorphic to $S$ as sets.
We denote by $\pi: Q\to S$ the (set-theoretic) bijection
that gives this isomorphism.
We define the ternary operation $\mu$ on 
$S$ by
\begin{equation}\label{pre:def:mu}
\mu(a, b, c)=\eta(b, a, c) \quad(a, b, c\in S).
\end{equation}
\begin{proposition}\label{pre:prop:mu}
The ternary operation $\mu$
satisfies\/$:$
\begin{align*}
&\mu(a, \mu(a, b, c), \mu(\mu(a, b, c), c, d))=\mu(a, b, \mu(b, c, d))
\quad(\forall a, b, c, d\in S);
\\
&\mu(\mu(a, b, c), c, d)=\mu(\mu(a, b, \mu(b, c, d)), \mu(b, c, d), d)
\quad
(\forall a, b, c, d\in S).
\end{align*}
\end{proposition}
For $\lambda\in Q$,
we define the map $\sigma(\lambda): Q\times Q\to Q\times Q$
by 
\[
\sigma(\lambda)(u, v)
=
(h(\lambda, v, u)\backslash((\lambda v)u), 
\lambda\backslash h(\lambda, v, u)).
\]
Here, $h(\lambda, v, u)=\pi^{-1}(\mu(\pi(\lambda), \pi(\lambda v),
\pi((\lambda v) u)))$
$(\lambda, u, v\in Q)$
and see \eqref{pre:eq:backslash}
for $h(\lambda, v, u)\backslash((\lambda v)u)$.
Proposition \ref{pre:prop:mu}
implies \eqref{pre:eq:avbr},
and
we have the following as a result.
\begin{proposition}\label{pre:prop:dybm}
The map $\sigma(\lambda)$ is a dynamical Yang-Baxter map.
\end{proposition}
Furthermore, 
we assume that this dynamical Yang-Baxter map $\sigma(\lambda)$
satisfies:
\begin{enumerate}
\item
the set $S$ is finite
(and so is the set $Q$);
\item
for any $b, c, d\in S$, there uniquely exists $a\in S$
such that
\begin{equation}\label{sufficient:uniquesol1}
\mu(a, b, c)=d;
\end{equation}
\item
for any $a, c, d\in S$, there exists a unique solution $b\in S$
to
$\eqref{sufficient:uniquesol1}$;
\item
for any $a, b, d\in S$, there exists a unique solution $c\in S$
to
$\eqref{sufficient:uniquesol1}$.
\end{enumerate}
This dynamical Yang-Baxter map $\sigma(\lambda)$
produces the Hopf algebroid $A_\sigma$
\cite[Sections 3 and 4]{s2016}.
We will briefly describe it as below. 

Let $\K$ be an arbitrary field,
and let $M_Q$ denote the $\K$-algebra
of all $\K$-valued maps on the set $Q$.
We define a map
$T_a: M_Q\to M_Q$ $(a\in Q)$ by
\begin{equation}\label{pre:eq:Ta}
T_a(f)(\lambda)=f(\lambda a)
\quad(f\in M_Q, \lambda, a\in Q).
\end{equation}
Let $L_{ab}, (L^{-1})_{ab}$
$(a, b\in Q)$ be indeterminates.
We define the set $AQ$ by
\[
AQ:=(M_Q\otimes_\K M_Q)\bigsqcup\{L_{ab}|a, b, \in Q\}\bigsqcup
\{(L^{-1})_{ab}|a, b, \in Q\}.
\]
$A_\sigma$ is the quotient of the free $\K$-algebra $\K\langle AQ\rangle$
on the set $AQ$ by two-sided ideal $I_\sigma$ whose generators are:
\begin{enumerate}
\item
$\xi+\xi'-(\xi+\xi')$,
$c\xi-(c\xi)$,
$\xi\xi'-(\xi\xi')$
$(c\in\K, \xi, \xi'\in M_Q\otimes_\K M_Q)$.

Here the symbol $+$ in $\xi+\xi'$ means the addition in the algebra 
$\K\langle AQ\rangle$,
while the symbol $+$ in $(\xi+\xi')(\in AQ)$ is the addition in the algebra
$M_Q\otimes_\K M_Q$.
The notations of the scalar products and products in the other generators 
are similar.
\item
$\displaystyle
\sum_{c\in Q}L_{ac}(L^{-1})_{cb}-\delta_{ab}\emptyset$,
$\displaystyle
\sum_{c\in Q}(L^{-1})_{ac}L_{cb}-\delta_{ab}\emptyset$
$(a, b\in Q)$.

Here $\delta_{ab}$ denotes  
Kronecker's delta symbol.
\item
$(T_{a}(f)\otimes 1_{M_Q})L_{ab}-L_{ab}(f\otimes 1_{M_Q})$,

$(1_{M_Q}\otimes T_{b}(f))L_{ab}-L_{ab}(1_{M_Q}\otimes f)$,

$(f\otimes 1_{M_Q})(L^{-1})_{ab}-(L^{-1})_{ab}
(T_{b}(f)\otimes 1_{M_Q})$,

$(1_{M_Q}\otimes f)
(L^{-1})_{ab}-(L^{-1})_{ab}
(1_{M_Q}\otimes T_{a}(f))$
$(f\in M_Q, a, b\in Q)$.

Here
$1_{M_Q}$
defined by $1_{M_Q}(\lambda)=1$ $(\lambda\in Q)$ is the unit
of $M_Q$
(for $T_a$, see \eqref{pre:eq:Ta}).
\item 
$\sum_{x, y\in Q}(\sigma^{xy}_{ac}\otimes1_{M_Q})L_{yd}L_{xb}
-
\sum_{x, y\in Q}(1_{M_Q}\otimes\sigma^{bd}_{xy})L_{cy}L_{ax}$
$(a, b, c, d\in Q)$.

Here
$\sigma^{xy}_{ac}\in M_Q$ is defined by
\[
\sigma^{xy}_{ac}(\lambda)
=\begin{cases}
1,&\mbox{if $\sigma(\lambda)(x, y)=(a, c)$};
\\
0,&\mbox{otherwise}.
\end{cases}
\]
\item
$\emptyset-1_{M_Q}\otimes 1_{M_Q}$.
\end{enumerate}
\begin{proposition}\label{pre:prop:hopfalgebroid}
This algebra $A_\sigma$ is a Hopf algebroid.
\end{proposition}
For details of the Hopf algebroid including its definition, 
see \cite{bohm2005,bohm2004}.

This Hopf algebroid $A_\sigma$ produces the rigid tensor category
$\mathrm{Rep}_{\mathcal{V}_G}(\sigma)_f$
consisting of finite-dimensional L-operators
associated with the dynamical Yang-Baxter map $\sigma$
\cite{s2016,stakeuchi}.
For the definition of $\mathrm{Rep}_{\mathcal{V}_G}(\sigma)_f$,
we need a tensor category $\mathcal{V}_G$.
Here, $G$ is the opposite group of the group of 
all permutations on the set $Q$.

An object of the category $\mathcal{V}_G$ is 
a $G$-graded $\K$-vector space $V=\oplus_{\alpha\in G}V_\alpha$;
and its morphism $f: V\to W$ is a map $f: Q\to\mathrm{Hom}_\K(V, W)$
satisfying
\[
f(\lambda)V_\alpha\subset\bigoplus_{\beta\in G, \beta(\lambda)
=\alpha(\lambda)}
W_\beta
\]
for any $\lambda\in Q$,
where $\mathrm{Hom}_\K(V, W)$ is the $\K$-vector space of $\K$-linear
maps from $V$ to $W$.
In addition, the composition $fg$ of morphisms $f$ and $g$
is defined by
$(fg)(\lambda):=f(\lambda)\circ g(\lambda)$
$(\lambda\in Q)$.

$\mathcal{V}_G$ is a tensor category.
The tensor product $V\otimes W$
of objects $V=\oplus_{\beta\in G}V_\beta$
and
$W\oplus_{\gamma\in G}W_\gamma$ is
$V\otimes W=\oplus_{\alpha\in G}(V\otimes W)_\alpha$,
where
$(V\otimes W)_\alpha:=\oplus_{\beta, \gamma\in G, \alpha=\gamma\beta}
V_\beta\otimes_\K W_\gamma$.
Here, $\gamma\beta(\in G)$ is the multiplication of $\gamma$
and $\beta$ in the group $G$.
In addition, 
the tensor product $f\otimes g$ of 
morphisms
$f: U\to V$
and
$g: W\to Y$
is 
a map
$(f\otimes g)(\lambda)=\sum_{\alpha\in G}
(\sum_{\beta, \gamma\in G, \alpha=\gamma\beta}
(f\otimes g)(\lambda)_{\beta, \gamma})$
$(\lambda\in Q)$.
Here,
$(f\otimes g)(\lambda)_{\beta, \gamma}\in\mathrm{Hom}_\K
(U_\beta\otimes_\K W_\gamma, V\otimes Y)$
is defined by
\[
(f\otimes g)(\lambda)_{\beta, \gamma}(u_\beta\otimes w_\gamma)
:=f(\gamma(\lambda))(u_\beta)\otimes g(\lambda)(w_\gamma)
\quad(u_\beta\in U_\beta, w_\gamma\in W_\gamma).
\]
The unit $I$ is the $\K$-vector space $\K=\oplus_{\alpha\in G}I_\alpha$
with
\[
I_\alpha=\begin{cases}
\K,&\mbox{if $\alpha=1_G(=\id_Q)$};\\
\{ 0\},&\mbox{otherwise}.
\end{cases}
\]
The left and right unit constraints with respect to this unit $I$
are defined by
$l_V(\lambda)=\oplus_{\alpha\in G}l_V(\lambda)_\alpha$
and
$r_V(\lambda)=\oplus_{\alpha\in G}r_V(\lambda)_\alpha$
$(\lambda\in Q)$.
Here,
\begin{align*}
&l_V(\lambda)_\alpha: (I\otimes V)_\alpha=I_{1_G}\otimes_\K V_\alpha
(=\K\otimes_\K V_\alpha)\ni a\otimes v\mapsto av\in V_\alpha\subset V;
\\
&r_V(\lambda)_\alpha: (V\otimes I)_\alpha=V_\alpha\otimes_\K I_{1_G}
(=V_\alpha\otimes_\K \K)\ni v\otimes a\mapsto av\in V_\alpha\subset V.
\end{align*}

We are now in a position to construct 
the rigid tensor category $\mathrm{Rep}_{\mathcal{V}_G}(\sigma)_f$.
For $a\in Q$, we define $\mathrm{deg}(a)\in G$
by
$\mathrm{deg}(a)(\lambda):=\lambda a$
$(\lambda\in Q)$.
This definition is unambiguous, because $Q$ is a quasigroup.
Let $\K Q$ denote the $\K$-vector space with the basis $Q$.
This $\K Q$ is an object of $\mathcal{V}_G$,
since $\K Q=\oplus_{\alpha\in G} (\K Q)_\alpha$
and
\[
(\K Q)_\alpha=
\begin{cases}
\K a(\cong \K),&\mbox{if $\exists a\in G$ such that 
$\alpha=\mathrm{deg}(a)$};\\
\{ 0\},&\mbox{otherwise.}
\end{cases}
\]
This is well defined on account of the definition of
the quasigroup.
We can regard the map $\sigma(\lambda): Q\times Q\to Q\times Q$
$(\lambda\in Q)$
as a $\K$-linear map on $\K Q\otimes_\K\K Q$,
and this $\sigma: \K Q\otimes \K Q\to \K Q\otimes \K Q$ is a morphism
of the category $\mathcal{V}_G$.
\begin{remark}\label{pre:rem:brvg}
A version of the quantum dynamical Yang-Baxter equation
\eqref{pre:eq:avbr} for the dynamical Yang-Baxter map
$\sigma(\lambda)$
is exactly the same as the braid relation in the
category $\mathcal{V}_G$
\cite[Proposition 4.5]{s2016}.
\end{remark}

The object of $\mathrm{Rep}_{\mathcal{V}_G}(\sigma)_f$ is a pair
of an object $V\in\mathcal{V}_G$ of finite dimensions
and an isomorphism $L_V: V\otimes \K Q\to\K Q\otimes V$
of $\mathcal{V}_G$ satisfying
\[
(\sigma\otimes\id_V)(\id_{\K Q}\otimes L_V)(L_V\otimes\id_{\K Q})
=
(\id_{\K Q}\otimes L_V)(L_V\otimes\id_{\K Q})(\id_V\otimes\sigma).
\]
The morphism $f: (V, L_V)\to(W, L_W)$ is a morphism
$f: V\to W$ of $\mathcal{V}_G$
such that
$(\id_{\K Q}\otimes f)L_V=L_W(f\otimes \id_{\K Q})$.

The tensor product $(V, L_V)\boxtimes(W, L_W)$
is $(V\otimes W, (L_V\otimes\id_W)(\id_V\otimes L_W))$,
and
the tensor product of morphisms is exactly the same as
that of $\mathcal{V}_G$:
$f\boxtimes g:=f\otimes g$.
The unit is $(I, r_{\K Q}^{-1}l_{\K Q})$.

From \cite[Section 5]{s2016},
we have the following.
\begin{proposition}\label{pre:prop:rigidtensor}
$\mathrm{Rep}_{\mathcal{V}_G}(\sigma)_f$ is a rigid tensor category.
\end{proposition}
\section{Ternary operations from s-sets}
Let $M$ be a non-empty set
and
$s_x: M\to M$
a bijection
for each $x\in M$.
An s-set $(M, \{ s_x\}_{x\in M})$ is a generalization of 
the regular s-manifold \cite{kowalski}
(see Introduction)
in the generalized symmetric spaces
from the algebraic point of view.
\begin{definition}\label{tos:def:sset}
A pair $(M, \{ s_x\}_{x\in M})$ is an s-set, iff
the bijections $s_x$ satisfy
\begin{equation}\label{tos:eq:symmetry}
s_x\circ s_y=s_{s_x(y)}\circ s_x
\end{equation}
for any $x, y\in M$.
\end{definition}
For a simple example, we note that any group $G$
makes an s-set $(G, \{ s_x\})$.
Here, the bijection $s_x$
$(x\in G)$ is defined by $s_x(y)=xyx^{-1}$
$(y\in G)$.

In this section,
we will show that 
every s-set $(M, \{ s_x\})$ can produce 
ternary
operations $\eta$ on the set $M$ satisfying $\eqref{hps:eq:eta1}$.

Let $I=(i_1, i_2, \ldots, i_l)\in\Z^l$ 
$(l\geq 1)$,
and we write 
\[
w_I(X, Y)=
\begin{cases}
X^{i_1}Y^{i_2}\cdots X^{i_{l-1}}Y^{i_l},&\mbox{if $l$ is even};\\
X^{i_1}Y^{i_2}\cdots Y^{i_{l-1}}X^{i_l},&\mbox{if $l$ is odd}.
\end{cases}
\]
Here, $w_I(X, Y)$ is an element of the quotient of the free algebra
on the set $\{ X, X^{-1}, Y, Y^{-1}\}$
by
the two-sided ideal whose generators are 
$XX^{-1}-1, X^{-1}X-1, YY^{-1}-1$,
and $Y^{-1}Y-1$.

Let $\eta_I$ denote the ternary operation on an s-set $(M, \{ s_x\})$ defined
by
\begin{equation}
\label{tos:eq:wI}
\eta_I(x, y, z)=w_I(s_x, s_y)(z)
\quad(x, y, z\in M).
\end{equation}
\begin{theorem}\label{hps:thm:eta1}
The ternary operation $\eta:=\eta_I$ satisfies $\eqref{hps:eq:eta1}$.
\end{theorem}
\begin{proof}
The proof 
is by induction on the length $l(w_I):=|i_1|+|i_2|+\cdots+|i_l|$ 
of the word $w_I(X, Y)$.

If $l(w_I)=0$, then $w_I$ is an empty word,
and
$\eta_I(x, y, z)=w_I(s_x, s_y)(z)=z$
as a result.
An easy computation shows \eqref{hps:eq:eta1}.

If $l(w_I)=1$, then $w_I(X, Y)=X, X^{-1}, Y, Y^{-1}$.
We give the proof only for the case that $w_I(X, Y)=X^{-1}$.
Because $w_I(X, Y)=X^{-1}$, $\eta_I(x, y, z)=s_x^{-1}(z)$.
By substituting $s_x^{-1}(u)$ into $y$ in \eqref{tos:eq:symmetry},
\begin{equation}\label{hps:eq:sxsu}
s_xs_{s_x^{-1}(u)}=s_us_x,
\end{equation}
and consequently,
the right-hand-side of \eqref{hps:eq:eta1} is $s_x^{-1}s_u^{-1}(w)$,
which is exactly the left-hand-side of \eqref{hps:eq:eta1}.

If $l(w_I)\geq 2$,
then there exists a word $w'(X, Y)$ whose length is less than
$l(w_I)$ such that
$w_I(X, Y)=Xw'(X, Y)$, or $w_I(X, Y)=X^{-1}w'(X, Y)$, or
$w_I(X, Y)=Yw'(X, Y)$, or $w_I(X, Y)=Y^{-1}w'(X, Y)$.
For example,
if $w_I(X, Y)=X^{-3}Y^2XY^8$,
then we set $w'(X, Y)=X^{-2}Y^2XY^8$,
which satisfies $w_I(X, Y)=X^{-1}w'(X, Y)$.

We prove only for the case that $w_I(X, Y)=X^{-1}w'(X, Y)$.
Since $w_I(X, Y)=X^{-1}w'(X, Y)$, 
$\eta_I(x, y, z)=s_x^{-1}w'(s_x, s_y)(z)$,
and the right-hand-side of \eqref{hps:eq:eta1} is
\begin{equation}\label{hps:eq:pr0}
s_{s_x^{-1}w'(s_x, s_y)(u)}^{-1}w'(s_{s_x^{-1}w'(s_x, s_y)(u)},
s_{s_x^{-1}w'(s_x, s_y)(v)})s_x^{-1}w'(s_x, s_y)(w).
\end{equation}
Substituting 
$w'(s_x, s_y)(u)$
into
$u$ in \eqref{hps:eq:sxsu}
gives
\begin{equation}\label{hps:eq:sformulas}
s_{s_x^{-1}w'(s_x, s_y)(u)}=s_x^{-1}s_{w'(s_x, s_y)(u)}s_x,
\end{equation}
and,
in the same manner, we can see that
$s_{s_x^{-1}w'(s_x, s_y)(v)}=s_x^{-1}s_{w'(s_x, s_y)(v)}s_x$.
Now 
\eqref{hps:eq:pr0} becomes
\begin{equation}\label{hps:eq:pr05}
s_{s_x^{-1}w'(s_x, s_y)(u)}^{-1}w'(s_x^{-1}s_{w'(s_x, s_y)(u)}s_x,
s_x^{-1}s_{w'(s_x, s_y)(v)}s_x)s_x^{-1}w'(s_x, s_y)(w).
\end{equation}
\begin{lemma}
Let $X, Y, Z, X^{-1}, Y^{-1}, Z^{-1}$
be indeterminates
satisfying 
\[
XX^{-1}=X^{-1}X=YY^{-1}=Y^{-1}Y=ZZ^{-1}=Z^{-1}Z=1,
\]
and let $w(X, Y)$ be a word of $X$, $X^{-1}$, $Y$, and $Y^{-1}$.
Then $w(ZXZ^{-1}, ZYZ^{-1})=Zw(X, Y)Z^{-1}$.
Here, we regard $w(X, Y)$ 
as an element of the quotient of the free algebra
on the set $\{ X, X^{-1}, Y, Y^{-1},$ $Z, Z^{-1}\}$
by
the two-sided ideal whose generators are $XX^{-1}-1$, $X^{-1}X-1$, 
$YY^{-1}-1$,
$Y^{-1}Y-1$, $ZZ^{-1}-1$, and $Z^{-1}Z-1$.
\end{lemma}
The proof of this lemma is obvious,
because $(ZX^{\pm 1}Z^{-1})^i=ZX^{\pm i}Z^{-1}$
for any $i\in\Z$.

On account of this lemma, \eqref{hps:eq:pr05} is
\begin{equation}\label{hps:eq:pr1}
s_{s_x^{-1}w'(s_x, s_y)(u)}^{-1}s_x^{-1}w'(s_{w'(s_x, s_y)(u)},
s_{w'(s_x, s_y)(v)})w'(s_x, s_y)(w).
\end{equation}
We define the ternary operation $\eta'$ on $M$ by
$\eta'(x, y, z)=w'(s_x, s_y)(z)$
$(x, y, z\in M)$.
Because of the fact that
\[
w'(s_{w'(s_x, s_y)(u)},
s_{w'(s_x, s_y)(v)})w'(s_x, s_y)(w)=\eta'(\eta'(x, y, u),
\eta'(x, y, v), \eta'(x, y, w)),
\]
the induction hypothesis,
and \eqref{hps:eq:sformulas},
\eqref{hps:eq:pr1} is
\begin{align}\nonumber
&s_{s_x^{-1}w'(s_x, s_y)(u)}^{-1}s_x^{-1}
\eta'(x, y, \eta'(u, v, w))
\\
=&\label{hps:eq:pr2}
s_x^{-1}
s_{w'(s_x, s_y)(u)}^{-1}w'(s_x, s_y)w'(s_u, s_v)(w).
\end{align}

We will prove the following claim later.
\begin{claim}\label{hps:claim:s}
For any $x, y, u\in M$,
\begin{equation}\label{hps:eq:s}
s_{w'(s_x, s_y)(u)}\circ w'(s_x, s_y)=w'(s_x, s_y)\circ s_u.
\end{equation}
\end{claim}
By virtue of this claim,
the right-hand-side of \eqref{hps:eq:pr2}
is
exactly the same as the left-hand-side of \eqref{hps:eq:eta1},
which is the desired conclusion.
\end{proof}
\begin{proof}[Proof of Claim $\ref{hps:claim:s}$]
The proof is similar to that of the theorem;
we will prove it 
by induction on the length $l(w')$ of the word $w'(X, Y)$.

The proof of the cases that $l(w')=0, 1$ is obvious.

If $l(w')\geq 2$,
then there exists a word $w''(X, Y)$ whose length is less than
$l(w')$ such that
$w'(X, Y)=Xw''(X, Y)$, or $w'(X, Y)=X^{-1}w''(X, Y)$, or
$w'(X, Y)=Yw''(X, Y)$, or $w'(X, Y)=Y^{-1}w''(X, Y)$.
We prove only for the case $w'(X, Y)=X^{-1}w''(X, Y)$.
Since $w'(X, Y)=X^{-1}w''(X, Y)$, 
the left-hand-side of \eqref{hps:eq:s} is
\begin{equation}\label{hps:eq:sw'}
s_{s_x^{-1}w''(s_x, s_y)(u)}s_x^{-1}w''(s_x, s_y).
\end{equation}
By substituting 
$w''(s_x, s_y)(u)$
into $u$ in \eqref{hps:eq:sxsu},
\eqref{hps:eq:sw'} is 
$s_x^{-1}s_{w''(s_x, s_y)(u)}w''(s_x, s_y)$.
By the induction hypothesis,
this is
$s_x^{-1}w''(s_x, s_y)s_u=w'(s_x, s_y)s_u$,
which is exactly
the right-hand-side of \eqref{hps:eq:s}.
This establishes the formula.
\end{proof}
\section{Dynamical Yang-Baxter maps from s-sets}
\label{hpss}
This section is devoted to the construction
of the dynamical Yang-Baxter maps (Definition \ref{pre:def:dybm})
via
the homogeneous pre-systems
(Definition \ref{pre:def:hps})
by means of suitable s-sets (Definition \ref{tos:def:sset}).

Let $R$ be a ring with the unit $1(\ne 0)$,
$M$ a left $R$-module,
and $r$
an invertible element of the ring $R$.
We define $s_x: M\to M$ $(x\in M)$ by
\[
s_x(y)=(1-r)x+r y\quad(y\in M).
\]
\begin{proposition}\label{hpss:prop:s-set}
$(M, \{ s_x\})$ is an s-set.
\end{proposition}
In fact, the inverse of the map $s_x$
is
\[
s_x^{-1}(y)=(1-r^{-1})x+r^{-1} y,
\]
and
it is a simple matter to show \eqref{tos:eq:symmetry}.

Let $I=(i_1, i_2, \ldots, i_l)\in\Z^l$
$(l\geq 2)$.
We denote by $\Phi_I(X)$ the following polynomial
of the variables $X$ and $X^{-1}$.
\begin{equation}
\label{hpss:eq:phi}
\Phi_I(X)=
\begin{cases}
1+\sum_{j=1}^{l}(-1)^jX^{\sum_{m=1}^ji_m},&\mbox{if $l$ is even};\\
1+\sum_{j=1}^{l-1}(-1)^jX^{\sum_{m=1}^ji_m},&\mbox{if $l$ is odd}.
\end{cases}
\end{equation}

\begin{proposition}\label{hpss:prop:eta}
$\eta_I(x, y, z)=(\Phi_I(r)-r^d)x+(1-\Phi_I(r))y+r^dz$
for any $x, y, z\in M$.
Here, $\eta_I(x, y, z)$ is defined in $\eqref{tos:eq:wI}$
and
$d:=i_1+i_2+\cdots+i_l$.
\end{proposition}
\begin{proof}
The proof of the proposition is by induction on $l$.
For the proof of the $l=2$ case, we need
\begin{lemma}
For any integer $i$,
\[
s_x^i(z)=(1-r^i)x+r^iz
\quad(\forall x, z\in M).
\]
\end{lemma}
As a corollary of this lemma,
\[
s_x^is_y^j(z)=(1-r^i)x+(r^i-r^{i+j})y+r^{i+j}z
\quad(i, j\in\Z, x, y, z\in M),
\]
which
immediately induces the $l=2$ case.
The rest of the proof is straightforward.
\end{proof}
As a corollary, we find
\begin{corollary}\label{hpss:cor:hps}
If the invertible element $r\in R$ satisfies
that $\Phi_I(r)=0$ in $R$, then $(M, \eta)$ is a homogeneous pre-system
$($see Definition $\ref{pre:def:hps}$$)$
satisfying $\eqref{pre:eq:hps}$.
\end{corollary}
The proof is obvious, since
\begin{equation}\label{hpss:eq:etaI}
\eta_I(x, y, z)=-r^dx+y+r^dz
\quad(x, y, z\in M),
\end{equation}
if $\Phi_I(r)=0$.
\begin{example}
Let $k(\geq 2)$ be a positive integer,
and
let $r(\in\C)$ be a primitive $k$-th root of unity.
We denote by $\Phi_k$ the cyclotomic polynomial
of level $k$.
If $k=6$, then
$\Phi_I(X):=\Phi_6(X)=1-X+X^2$
satisfies \eqref{hpss:eq:phi}
$(I=(1, 1))$.
Because $\Phi_I(r)=0$,
any $\C$-vector space $V$ produces
a homogeneous pre-system
satisfying $\eqref{pre:eq:hps}$
on account of Corollary \ref{hpss:cor:hps}.

If $k=10, 14, 15, 18,
20, 21, 22, 24, 26, 28, 33, 34, 35, 36, 38, 39, 40$,
then any primitive $k$-th root $r$ of unity can also give 
birth to a homogeneous pre-system
satisfying $\eqref{pre:eq:hps}$.
\end{example}
\begin{remark}
The s-set
$(V, \{ s_x\})$ 
for any finite-dimensional $\C$-vector space $V$
in the above example 
is
a regular s-manifold of order $k$
(see \cite[Definition II.43]{kowalski}).
Here, we regard $V$ as an $\R$-vector space of $2\dim_\C V$ dimensions.
\end{remark}
\begin{example}\label{hpss:example:finites-set}
Let $R:=\Z/5\Z$.
The invertible element $r:=2\in R$ is
a root of $\Phi_{(2, 1)}(X)=1-X^2+X^3$,
and hence $r$ with any (left) 
$R$-module $M$ gives a homogeneous pre-system
satisfying $\eqref{pre:eq:hps}$ according to Corollary \ref{hpss:cor:hps}.
In fact, $\eta_{(2, 1)}(x, y, z)=-3x+y+3z$.
In a similar fashion,
the invertible element
$3\in R$ also provides with
such a homogeneous pre-system,
because
$3$ is a root of $\Phi_{(2, -1)}(X)=1-X^2+X$.
\end{example}
\begin{example}\label{hpss:example:ordinarymu}
Let $R:=\Z/5\Z$.
The invertible element $r:=2\in R$ is
also a root of $\Phi_{(2, 1, 1)}(X)=\Phi_{(2, 1)}(X)=1-X^2+X^3$,
and hence $r$ 
with any (left) $R$-module $M$ gives a homogeneous pre-system
satisfying $\eqref{pre:eq:hps}$ according to Corollary \ref{hpss:cor:hps}.
In this case, $\eta_{(2, 1, 1)}(x, y, z)=-x+y+z$,
because $r^d=r^4=1$ in $R$.
On account of \eqref{pre:def:mu},
this $\eta_{(2, 1, 1)}$ yields the ternary operation $\mu$ 
in \cite[$(4.13)$]{s2016}.
\end{example}

Let $Q$ be a quasigroup
(Definition \ref{pre:def:quasigroup}), isomorphic to
the set $M$ as sets.
If the element $r(\in R)$ satisfies
$\Phi_I(r)=0$,
then it follows 
from Proposition \ref{pre:prop:dybm}
and Corollary \ref{hpss:cor:hps}
that this s-set $(M, \{ s_x\})$ with the quasigroup $Q$
gives birth to
the dynamical Yang-Baxter map
$\sigma(\lambda)$
(Definition \ref{pre:def:dybm}).
\section{Hopf algebroids from finite s-sets
and rigid tensor categories}
Let $R$ be a ring with the unit $1(\ne 0)$,
$M$ a left $R$-module,
$r(\in R)$
an invertible element
satisfying
$\Phi_I(r)=0$
(for $\Phi_I$, see \eqref{hpss:eq:phi}),
and
$Q$ a quasigroup, isomorphic to
the set $M$ as sets
(see Definition \ref{pre:def:quasigroup}).

In this last section, we will construct
Hopf algebroids by means of the dynamical Yang-Baxter maps
$\sigma(\lambda)$ in Section \ref{hpss}.
In order to define the Hopf algebroid,
we restrict our attention to the case that 
$M$ is finite
and that $|M|>1$
(it follows that $|Q|>1$).
We remark that 
$R=\Z/5\Z$ in Examples \ref{hpss:example:finites-set}
and \ref{hpss:example:ordinarymu}
is a field and that any (left) $R$-module $M$ is an $R$-vector space.
Hence, each $M(\ne\{ 0\})$ in Examples \ref{hpss:example:finites-set}
and \ref{hpss:example:ordinarymu}
of finite dimensions is finite and satisfies $|M|>1$.
\begin{proposition}\label{hafsrtc:prop:mu}
The ternary operation $\mu$ $\eqref{pre:def:mu}$
on $M$
satisfies\/$:$
\begin{enumerate}
\item[$(1)$]
for any $b, c, d\in M$, there exists a unique solution $a\in M$
to
$\eqref{sufficient:uniquesol1}$\/$;$
\item[$(2)$]
for any $a, c, d\in M$, there exists a unique solution $b\in M$
to
$\eqref{sufficient:uniquesol1}$\/$;$
\item[$(3)$]
for any $a, b, d\in M$, there exists a unique solution $c\in M$
to
$\eqref{sufficient:uniquesol1}$.
\end{enumerate}
\end{proposition}
The proof is clear from \eqref{hpss:eq:etaI}
and 
the fact that the element $r$ is invertible in the ring $R$.

By virtue of Proposition \ref{pre:prop:hopfalgebroid},
we have the following.
\begin{theorem}
$A_\sigma$ is a Hopf algebroid.
\end{theorem}
The homogeneous pre-system
$\eta_{(2, 1, 1)}$ in Example \ref{hpss:example:ordinarymu}
yields a Hopf algebroid $A_\sigma$ in \cite{s2016}.

Furthermore, we have the following from Proposition \ref{pre:prop:rigidtensor}.
\begin{theorem}
$\mathrm{Rep}_{\mathcal{V}_G}(\sigma)_f$
is a rigid tensor category.
\end{theorem}
In view of Proposition \ref{hafsrtc:prop:mu} (2)
(see \cite[Proposition 4.5]{s2016}),
$(\K Q, \sigma)\in\mathrm{Rep}_{\mathcal{V}_G}(\sigma)_f$,
and
this object
is not the unit of $\mathrm{Rep}_{\mathcal{V}_G}(\sigma)_f$,
because $|Q|>1$.
The rigid tensor category 
$\mathrm{Rep}_{\mathcal{V}_G}(\sigma)_f$
is hence non-trivial
(see \cite[Remark 3.9]{stakeuchi}).
\section*{Acknowledgments}
The second author was supported in part by KAKENHI (26400031).

\end{document}